\documentclass[runningheads]{llncs}
\usepackage[T1]{fontenc}
\usepackage{amsmath,amssymb}
\usepackage{graphicx}
\usepackage{wrapfig}
\usepackage{multicol}
\usepackage[dvipsnames]{xcolor}
\usepackage{faktor}
\usepackage[mode=buildnew]{standalone}
\usepackage[hidelinks]{hyperref}
\usepackage[capitalise,nameinlink]{cleveref}

\newcommand{\tuple}[1]{\overline{#1}}

\newcommand{\limplies}{\rightarrow}

\newcommand{\pres}[2]{\langle #1 \mid #2 \rangle}

\newcommand{\Z}{\mathbb{Z}}

\newcommand{\cF}{\mathcal{F}}

\newcommand{\cL}{\mathcal{L}}
\renewcommand{\restriction}{\mathord{\upharpoonright}}
\newcommand{\into}{\hookrightarrow}

\newcommand\sqsubsetsim{\mathrel{%
  \ooalign{\raise0.2ex\hbox{$\sqsubset$}\cr\hidewidth\raise-0.8ex\hbox{\scalebox{0.9}{$\sim$}}\hidewidth\cr}}}

\spnewtheorem{rems}{Remarks}{\itshape}{\normalfont}
\spnewtheorem*{claimproof}{Proof of Claim}{\itshape}{\rmfamily}
\spnewtheorem*{proofsketch}{Proof sketch}{\itshape}{\rmfamily}

\begin{document}

    \title{Comparing the Effective Content of Subshifts}

    \author{A. {Nakid Cordero}\orcidID{0009-0005-9908-5642} \and
    {I. Scott}\orcidID{0000-0003-4231-9741}}
    
    \authorrunning{A. {Nakid Cordero} and I. Scott}
    
    \institute{University of Wisconsin--Madison, Madison WI 53705, USA}

\maketitle     

    \begin{abstract}
    
        Motivated by Ziegler's computability-theoretic characterisation of finite absolute presentability between groups, we prove an analogous theorem in symbolic dynamics.  We introduce the notion of one subshift being \emph{finitely determined over} another and show that this relation is characterised by \emph{Ziegler reducibility} between their co-languages. We further investigate a notion of \emph{existential closure} for subshifts.
    
    \end{abstract}

\section{Introduction}

The study of multidimensional subshifts---colourings of $\Z^d$ that avoid a set of forbidden finite patterns---has greatly benefited from the tools of computability theory. A signature example is the simulation theorem by Hochman \cite{hochman_dynamics_2009} and later refinements by Aubrun--Sablik \cite{aubrun_simulation_2013}  and Durand--Romashchenko--Shen \cite{durand_effective_2010}, which state that every effectively closed shift is simulated by a subshift of finite type (SFT) of higher dimension.

In this vein, we characterise in computability-theoretic terms when a subshift is ``completely determined'' by another subshift using a finite amount of extra information. Namely, we prove that this relationship coincides with \emph{Ziegler reducibility}, denoted $\leq^*$, (see \cref{def:Ziegler red}) between their co-languages.

The definition of $\leq^*$ dates back to Ziegler's analysis of the finitely generated subgroups of existentially closed groups \cite{ziegler_algebraisch_1980}. This fits into a growing ``dictionary'' between notions and results in combinatorial group theory and in symbolic dynamics, as developed in \cite{jeandel_enumeration_2017,jeandel_characterization_2019}.
We also raise an issue with extending the dictionary further, to capture existentially closed groups.  We show that the most natural definition and construction of ``existentially closed subshifts'' (following the theory for groups) does not satisfy Ziegler's Theorem \cite[Folgerungen III.1.8]{ziegler_algebraisch_1980}.

\section{Preliminaries}

In what follows, we will be working in subspaces of $\Sigma^{\Z^d}$ where $\Sigma$ is a finite set and $d < \omega$.  Recall that this space comes endowed with a natural topology generated by the basic clopen sets $[w] = \left\{ x \in \Sigma^{\Z^d} \mid x \restriction_{\operatorname{dom}(w)} = w \right\}$,
where $w: D \to \Sigma$ is a map from a \emph{rectangle} $[m_0,n_0] \times \cdots \times [m_{d-1},n_{d-1}] \subseteq \Z^d$.  We write $\Sigma^{*d}$ for the set of such maps and call them \emph{($d$-dimensional) words}, when there is no risk of ambiguity, we simply write $\Sigma^*$.  For $k < d$, we identify $\Z^k \subseteq \Z^d$ via the map $(v_0, \ldots, v_{k-1}) \mapsto (v_0, \ldots, v_{k-1}, 0, \ldots, 0)$, and hence, with this convention $\Sigma^{*k} \subseteq \Sigma^{*d}$. It is a standard fact that $\Sigma^{\Z^d}$ with this topology is a compact space.  

The \emph{shift map} $\mathrm{sh}: \Z^d \times \Sigma^{\Z^d} \to \Sigma^{\Z^d}$ is given by $\mathrm{sh}( v, x)(u) =  x(u +  v)$. 

\begin{definition}
    A \emph{subshift} $S$ is a topologically closed, shift-invariant subset of $\Sigma^{\Z^d}$, for some set $\Sigma$, called the \emph{alphabet of $S$} and $d < \omega$ called the \emph{dimension of $S$}.  Unless otherwise specified, we will assume that $\Sigma$ is finite.  An element $x \in S$ is called a \emph{point} or \emph{configuration} of $S$.
    
    We write $\Sigma_S$ and $d_S$ for the alphabet and dimension of $S$, respectively.  
\end{definition}

    Because subshifts are closed and shift-invariant, we may identify $S$ with the set of words that induce open sets disjoint from $S$.

\begin{definition}
    The \emph{co-language} of a subshift $S$ is
    \[ \cL^c(S) = \{w: D \to \Sigma_S \mid D\ \text{is a rectangle in $\Z^{d_S}$ and } \forall x \in S\ (x \restriction_{\operatorname{dom}(w)} \neq w) \}. \]
    The \emph{language} of $S$ is $\cL(S) = \Sigma^{*d} \setminus \cL^c(S)$.  Equivalently, the language of $S$ is the set of words that appear in some $x \in S$.  For $x \in S$, we define $\cL(x)$ and $\cL^c(x)$ accordingly.
\end{definition}

    As part of the philosophy of this paper of translating between notions in group theory and symbolic dynamics, and following the lead of \cite{jeandel_characterization_2019}, we write $\pres{\Sigma_S}{F}^{d_S}$ to denote the largest (under inclusion) subshift $S$ such that $F \subseteq \cL^c(S)$.
    Further, we denote by
    \[ S \models L(G) \land L^c(B) \]
    that $G \subseteq \cL(S)$ and $B \subseteq \cL^c(S)$.  If $G = \{v\}$ and $B = \{v\}$ then we write $S \models L(v) \land L^c(w)$.

    If $S$ can be expressed as $\pres{\Sigma_S}{F}^{d_S}$ for $F$ finite, we call $S$ a \emph{subshift of finite type (SFT)}.  Similarly, if $F$ can be chosen computable (equivalently, c.e.), we say that $S$ is \emph{effectively closed}.

\begin{example}
    The \emph{full shift} on alphabet $\Sigma$ and dimension $d$ is the shift $\pres{\Sigma}{\emptyset}^d$.  Notice that every shift is a subset of the full shift in the same alphabet and dimension.

    On the other end of this spectrum, a \emph{minimal shift} is a shift $S$ such that if $T \subseteq S$ is a subshift, then $T$ is empty or equal to $S$.  Equivalently, $S$ is minimal if for every $x,y \in S$, we have $\cL(x) = \cL(y)$.
\end{example}

Given a subshift $S$, we have several ways to form a new subshift from $S$.

\begin{definition}\label{def:new subshifts from old}
    Let $S$ be a subshift.  The \emph{slice of $S$ to alphabet $\Delta \subseteq \Sigma_S$} is the subshift $S \restriction_\Delta = S \cap \Delta^{\Z^{d_S}}$.  The \emph{slice of $S$ to dimension $d < d_S$} is the subshift 
    \[S\restriction_d=\{x\in \Sigma_S^{\Z^d}\mid \exists y \in S\;\forall v \in \Z^d \;(x(v)=y(v,0,\dots, 0) \}.\]
    In the other direction, we form a $d_S + k$-dimensional subshift from $S$ as
    \[ S^{(k)} = \{x \in \Sigma_S^{\Z^{d_S + k}} \mid \exists y \in S\ (x(u_1, \ldots, u_{d_S + k})) = y(u_1, \ldots, u_{d_S}) \}. \]
\end{definition}

Notice that the two notions of slice do not commute; in general $(S\restriction_\Sigma)\restriction_d\subseteq (S\restriction_d)\restriction_\Sigma$.  

Given a subshift $S$ and a set of words $F$, we can also build a new subshift by enlarging the co-language.  We write $S/F$ for the subshift $\pres{\Sigma_S}{\cL^c(S) \cup F}^{d_S}$ and note that this is a \emph{subset} of $S$.

\smallskip

There are a number of ways in the literature of recognising a subshift as a ``subsubshift'' of another.  We mention two that will be relevant to this manuscript.  

\begin{definition}\label{def:factor map}
    If $T$ and $S$ are subshifts of dimension $d$, then $f: T \to S$ is a \emph{factor map} if $f$ is surjective and commutes with the shift action.

    If, in addition, $f$ is injective, $f$ is called a \emph{conjugacy} or an \emph{isomorphism}.
\end{definition}

By the Curtis--Hedlund--Lyndon theorem, a factor map $f: T\to S$ is given by a \emph{block code}  $\hat f: \Sigma_T^D\to\Sigma_S$ such that $f(x)(u)= \hat f(\operatorname{sh}(u,x)\restriction_ D)$. By passing to an isomorphic copy of $T$, we can always assume that $\hat{f}:\Sigma_T\to\Sigma_S$. For simplicity of notation, when it makes sense, we identify $\hat f$ with its extension to a map from $\Sigma_T^*$ to $\Sigma_S^*$.  In particular, for $w \in \Sigma_S^*$, we always have that $\hat f^{-1}(w)$ is well-defined and finite.

We also borrow a notion from \cite{jeandel_characterization_2019}, better adapted to our approach of identifying a shift with its co-language.

\begin{definition}\label{def:sqare-embedding}
    Let $S$ and $T$ be subshifts.  Then \emph{$S$ is a full restriction of $T$}, written $S \sqsubseteq T$, if $\cL(T) \cap \Sigma^{*d} = \cL(S)$.  In particular, if $S \sqsubseteq T$, we must have $\Sigma_S \subseteq \Sigma_T$ and $d_S \leq d_T$.  
    
    To allay this dependence on the alphabet, we write $S \sqsubsetsim T$ if there is an injective map $\hat f: \Sigma_S \into \Sigma_T$ such that $f(S) \sqsubseteq T$.    
\end{definition}

Observe that $\sqsubseteq$ is transitive and that $S \sqsubseteq T$ iff $T \models L(\cL(S)) \land L^c(\cL^c(S))$.

\section{A characterisation of computable subshifts}

    The key definition is that of ``finitely determined over'' (\cref{def:finitely determined} below).  The definition should capture the intuition that from $T$, together with a finite amount of extra information, one can determine $S$.  So, what does it mean to have access to $T$?  How will $S$ be presented?  And what kind of finite information are we allowed?
    
    We identify $T$ and $S$ with their co-languages (and also implicitly with their alphabets and dimensions).  Thus, having access to $T$ should correspond to ``coding'' the co-language of $T$ (assuming that we know the alphabet and dimension of $T$). This is captured by Jeandel--Vanier's notion of full restriction (\cref{def:sqare-embedding}).  Similarly, ``determining $S$'' will correspond to the resulting subshift having $S$ as a full restriction.  Thus, ``$S$ is finitely determined over $T$'' should have the following flavour:
    \[ \text{If $T \sqsubseteq R$ and $R$ satisfies (*), then $S \sqsubseteq R$}, \]
    where (*) is some finite set of conditions, not dependent on $T$ or $S$.
    
    Here the analogy with groups provides some insight.  The relevant condition in the group case is that of a group $G$ being \emph{$\exists_1$-isolated} over another group $H$ (\cite[Bem.\ III.1.6.2]{ziegler_algebraisch_1980}).
    
    \begin{definition}\label{def:exists-isolated}
        A group $G$ is \emph{$\exists_1$-isolated}\footnote{Ziegler called this relation \emph{finitely absolutely presented over $H$} (see \cite[Bem.\ III.1.6.2]{ziegler_algebraisch_1980})  Unfortunately this conflicts with the standard notation in group theory that a presentation of a group corresponds to the unique group on the generators which is ``maximally free'' subject to the relations (and irrelations).  Indeed, in general even if $F \leq G$, one will need to take a larger group $F$ in order to ``present'' $G$ over $H$.  Thus, in the absence of better notation, we choose the hopefully-explanatory ``$\exists_1$-isolated''.} over $H$ if there is an existential formula $\varphi(\tuple x, \tuple h)$ such that any $F \models \varphi(\tuple a, \tuple h)$ with $\langle \tuple h \rangle \cong H$ has $\langle \tuple a \rangle \cong G$.
    \end{definition}

    Thus we should require that (*), in the framework above, enforces finitary constraints on the language and co-language of $R$.  
    
    The next condition on $R$ arises from the observation that the notion of ``substructure'' in group theory (and indeed more generally in algebra or model theory) bifurcates for subshifts: the first notion (that $T$ is a substructure of $R$ if the atomic diagram of $T$ is a subset of the atomic diagram of $R$) is captured by $\sqsubseteq$, and this is what we have used so far.  However, there is a second notion; namely that there is a ``structure-preserving'' injection $f: T \to R$.  In subshift theory, the analogue to this notion is that of a factor map; thus $T$ is a ``substructure'' of $R$ if there is a surjective, shift-commuting map $f: R \to T$ (see \cref{def:factor map}).  Indeed, this is borne out by the Simulation Theorems---shift-theoretic analogues of Higman's Embedding Theorem in group theory---proved by Durand--Romashchenko--Shen (Theorem 1 of \cite{durand_effective_2010}) and Aubrun--Sablik (Theorem 3.1 of \cite{aubrun_simulation_2013}).  We shall need both notions of substructure.
    
    Finally, for this to all be well-defined, we need to fix the dimension of $R$.
    
    \begin{definition}\label{def:finitely determined}
        Let $S$ and $T$ be subshifts.  $S$ is \emph{finitely determined over $T$} if there are:
        \begin{itemize}
            \item A finite alphabet $\Sigma$
            \item A dimension $d$
            \item A pair of finite sets of words $(G,B)$ in $\Sigma$
            \item A block code $\hat f: \Sigma^D \to \Delta$, for some $\Delta \supseteq \Sigma_S \cup \Sigma_T$
        \end{itemize}
        such that 
        \begin{enumerate}
            \item There is some subshift $R$ in dimension $d$ and alphabet $\Sigma$ such that $G \subseteq \cL(R)$, $B \subseteq \cL^c(R)$, and $T \sqsubseteq f(R)$.
            \item Any such subshift $R$ further satisfies that $S \sqsubseteq f(R)$.
        \end{enumerate} 
        In this case, say that $S$ is \emph{finitely determined over $T$ by $(\Sigma, d, G, B, \hat f)$}.  
        
        If $S$ is finitely determined over the empty subshift, we say that $S$ is \emph{finitely determined}.
    \end{definition}

\begin{proposition}
\label{example:minimal sft fin det}
    Every minimal SFT is finitely determined.
\end{proposition}

\begin{proof}
    Let $S=\pres{\Sigma_S}{F}^{d_S}$ be minimal.  We claim that $S$ is finitely determined by $(\Sigma_S,d_S,\{w\},F,\mathrm{id})$, where $w$ is any word in $\cL(S)$.  Indeed, if $R$ is any subshift satisfying \cref{def:finitely determined} for this tuple, then $R$ is a nontrivial subset of $S$ and hence by minimality is equal to $S$.  Thus $\mathrm{id}(R) \sqsupseteq S$.
\qed\end{proof}

    In contrast to \cref{example:minimal sft fin det}, not every SFT is finitely determined over any subshift. We show in \cref{Zieglers thm full} that $S$ is finitely determined over $T$ if and only if $\cL^c(S)$ is Ziegler-reducible to $\cL^c(T)$.  In this section, we warm up by showing that $\cL^c(S)$ is computable if and only if $S$ is finitely determined over $\emptyset$. This result can be seen as an analogue  to Rips' theorem that a group has computable word problem if and only if there is a $\exists$-formula $\varphi$ which defines it \cite{rips_characterization_82}.
    
    \begin{theorem}\label{char of computable subshifts}
        Let $S$ be a subshift in a finite alphabet.  Then $\cL^c(S)$ is computable if and only if $S$ is finitely determined over $\emptyset$.
    \end{theorem}
    
    \begin{proof}
        Firstly, let $S$ be a $d_S$-dimensional computable subshift.  By \cite[Corollary 27]{jeandel_characterization_2019}, there is a minimal SFT $T$ in dimension $d_S + 2$ with $S \sqsubseteq T$.  Then $M$ is finitely determined by \cref{example:minimal sft fin det}, and hence $S$ is also finitely determined.

        For the other direction, suppose that $S$ is finitely determined over $\emptyset$ by $(\Sigma, k, G, B, \hat f)$.  Then for any subshift $R \subseteq \Sigma^{\mathbb{Z}^k}$ with $R \models L(G) \land L^c(B)$, we have that $S \sqsubseteq f(R)$; in particular $S \sqsubseteq f(\pres{\Sigma}{B}^k)$.  Thus, $w \in \cL^c(S)$ if and only if $f^{-1}(w) \subseteq \cL^c(\pres{\Sigma}{B}^k)$.  But $\cL^c(\pres{\Sigma}{B}^k)$ is c.e.\ since $B$ is finite, so it follows that $\cL^c(S)$ is c.e.
    
        On the other hand, $w \in \cL(S)$ if and only if $f(\pres{\Sigma}{B}^d)/\{w\} \not\models L(G)$.  Since $G$ and $B$ are finite, this is also c.e.
    \qed\end{proof}

\section{Relativising the characterisation}
In this section, we show that \cref{char of computable subshifts} relativises in the Ziegler-degrees, defined below.  The definition is unusual at first pass, so we invite the reader to consult \cref{Ziegler-red explanation} for a discussion.

\begin{definition}\label{def:Ziegler red}
    Let $(D_u)_{u < \omega}$ be an effective enumeration of the finite subsets of $\omega$ and $(D^1_v)_{v < \omega}$ an effective enumeration of the subsets of $\omega$ of cardinality at most 1 (e.g., $D^1_v = \{v - 1\} \cap \omega$).

    Let $X, Y \subseteq \omega$.  Then  \emph{$X$ is enumeration-reducible to $Y$}, denoted $X \leq_e Y$,  if there is a c.e.\ set $W$ such that \[ n \in X \iff \exists u\ (\langle n,u \rangle \in W \land D_u \subseteq Y). \]
    \emph{$X$ is 1-enumeration-reducible to $Y$}, $X \leq_e^1 Y$, if there is a c.e.\ set $W$ such that \[ n \in X \iff \exists u,v\ (\langle n,u,v \rangle \in W \land D_u \subseteq Y \land D^1_v \subseteq Y^c). \]
    Finally, \emph{$X$ is Ziegler-reducible to $Y$}, $X \leq^* Y$, if $X \leq_e Y$ and $X^c \leq_e^1 Y$.
\end{definition}

\begin{example}\label{Ziegler-red explanation}
    Since we can give an effective list $\{w_i\}_{i\in\omega}$ of all words on alphabet $\Sigma$, we can identify a word $w$ with its index. Now, for every $B \subseteq \Sigma^*$, we have $\cL^c(\pres{\Sigma}{B}^d) \leq_e B$. The reduction is given by the c.e.\ set \[W=\{\langle w,u\rangle\mid  \exists n\; (\mbox{every extension }v \mbox{ of }w\mbox{ of length }n \mbox{ contains a word from }D_u ) \}\]
\end{example}

Intuitively, $X \leq^* Y$ if and only if positive information about $X$ (i.e., statements of the form $n \in X$) can be determined by asking finitary positive questions about $Y$, while negative information can be determined by asking a series of questions involving at most one negative query and finitely many positive queries.  The notion was introduced by Martin Ziegler in \cite{ziegler_algebraisch_1980}, in which he showed that it is closely related to parameter-definability in groups.

\begin{theorem}\label{Zieglers thm full}
    Let $S$ and $T$ be subshifts in finite alphabets.  Then $S$ is finitely determined over $T$ if and only if $\cL^c(S) \leq^* \cL^c(T)$.
\end{theorem}

We prove the two directions independently.

\subsection{Finitely determined implies Ziegler reducible}

\begin{lemma}
    If $S$ is finitely determined over $T$, then $\cL^c(S) \leq^* \cL^c(T)$.
\end{lemma}

\begin{proof}
Let $S$ be finitely determined over $T$ via $(\Sigma, d, G, B, \hat{f})$, 
and let $R$ be a $d$-dimensional subshift over $\Sigma$ with $B \subseteq \cL^c(R)$, $G \subseteq \cL(R)$, and $T \sqsubseteq f(R)$.

Let $U = \hat f^{-1}(\cL^c(T))$
and set $Q = \pres{\Sigma}{B \cup U}^d$.  
Then, since $T \sqsubseteq f(R)$, we have $U \cap \cL(R) = \emptyset$ and $R \subseteq Q$, so $f(R) \subseteq f(Q)$. Hence $Q \models L(G) \land L^c(B)$, and $T \sqsubseteq f(Q)$.  Since this determines $S$ over $T$, we have $S \sqsubseteq f(Q)$.  
Furthermore, since $B$ is finite and $\hat f$ is a finite block code, $\cL^c(Q) \leq_e U \leq_e \cL^c(T)$. 

Because $\hat{f}$ is a block code, $\cL(f(Q)) \leq_e \cL(Q)$. Putting this together, we get
\[ \cL^c(T) \geq_e \cL^c(Q) \geq_e \cL^c(f(Q)) \geq_e \cL^c(S).\]

\smallskip

To show that $\cL(S) \leq_e^1 \cL^c(T)$, let $w\in\Sigma_S^{*d_S}$ and 
$U_w = \hat f^{-1}(w)$.  (Since $\hat f$ is finite-to-one, this will be a finite set of words in $\Sigma^{*d}$.)  Thus, $w \in \cL^c(S)$ if and only if $U_w \subseteq \cL^c(Q)$.  

If $w \in \cL^c(S)$, then $Q/U_w = Q$ and so $\cL^c(f(Q/U_w)) \cap (\cL(T) \cup G) = \emptyset$. 
On the other hand, if $w \in \cL(S)$, then $S \not\sqsubseteq f(Q/U_w)$, so if $G \subseteq \cL(Q/U_w)$, we must have $T \not\sqsubseteq f(Q/U_w)$.  Since $f(Q/U_w) \subseteq f(Q)$, this must be because some $u \in \cL(T)$ is in the co-language of $f(Q/U_w)$.  Hence, if $w \in \cL(S)$, then $\cL^c(f(Q/U_w)) \cap (\cL(T) \cup G) \neq \emptyset$.

Observe that $\cL^c(f(Q/U_w)) \leq_e \cL^c(T)$, since $U_w$ is finite and $\cL^c(f(Q)) \leq_e \cL^c(T)$, as argued above.
Thus $\cL(S) \leq^1_e \cL^c(T)$ via the following algorithm for each $w$: Using $\cL^c(T)$, enumerate $\cL^c(f(Q/U_w))$.  For each element, check if it is in $G$ (since $G$ is finite, this can be built into the $1$-enumeration operator) or $\cL(T)$.  The latter check requires information about a single bit of $\cL(T)$ at a time.  Thus, $\cL(S) \leq_e^1 \cL^c(T)$, and so $\cL^c(S) \leq^* \cL^c(T)$.
\qed\end{proof}

\subsection{Ziegler reducible implies finitely determined}

We now complete the proof of \cref{Zieglers thm full}.

\begin{lemma}
\label{thm: second half of Ziegler}
    Let $S$ and $T$ be subshifts with $\cL^c(S) \leq^* \cL^c(T)$.  Then $S$ is finitely determined over $T$.
\end{lemma}

The proof relies heavily on the construction of a subshift which acts as an oracle for the language of a given subshift.  This \emph{oracle shift} was introduced in the proof of Theorem 23 in \cite{jeandel_characterization_2019}, their analogue of the relativised Higman Embedding Theorem.  The other key ingredient is the Simulation Theorem in \cite{durand_effective_2010} and \cite{aubrun_simulation_2013}, which shows that every effectively closed shift of dimension $d$ is the $d$-dimensional slice of a factor of a $d+1$-dimensional SFT.

\subsubsection*{The oracle shift}
\label{sec:oracle}

Since it was sketched in \cite{jeandel_characterization_2019}, we provide details of the construction here.  It may be useful to also reference the proof of Theorem 23 in \cite{jeandel_characterization_2019}.  The idea of the oracle shift is that, given a $d$-dimensional shift $S$ in a finite alphabet, it effectively produces a $d+1$-dimensional shift $O_S$ that displays the language of $S$ in a uniform way.  The forbidden words of $O_S$ can be chosen as $F \cup \cL^c(S)$, where the set $F$ is c.e.  Further, $\cL(S) \leq_1 O_S$---and in fact, a stronger relation holds: $S \sqsubseteq O_S$.

Let $S$ be a $d$-dimensional subshift in the alphabet $\Sigma$ and let $\Sigma_{O_S} = \Sigma \cup \{\#\}$.  

The subshift $O_S$ can be thought of as a ``stack'' of $d$-dimensional shifts, which we, borrowing the terminology of the $d= 1$ case, call ``rows''.  Each $d$-dimensional row in this stack will be associated with a number $i \in \omega \cup \{\infty\}$ called the \emph{type} of the row.  Rows of type $i$ consist of words $w \in \cL(S)$ which are on hypercubes of length $2^{i}-1$, surrounded by a single layer of $\#$s on each side.  
These rows will be periodic, with period $2^{i} \cdot |\Sigma|^{2^{i}}$ in each of their $d$ dimensions.  

Further, the rows interact.  Each of the rows of type $i$ are identical, and appear every $2^{i}$ rows.  Thus every other row is a row of type $1$, and every other of the remaining rows is a row of type $2$, etc.  Moreover, each connected word in $\Sigma$ on a row of type $i$ appears as a subword of some word in the rows of type $i+1$.

The construction of $O_S$ is then essentially the following: first $O_{\pres{\Sigma}{}}$ is constructed, which has c.e.\ set of forbidden words.  $O_S$ is then obtained by further forbidding $\cL^c(S)$.

Say that a word $d$-dimensional word $w$ on alphabet $\Sigma_{O_S}$ is a \emph{valid pattern of type $i$} if its domain is a hypercube of length $2^i+1$ where the hypercube of length $2^i-1$ in the interior only has letters from $\Sigma$ and it is wrapped by a layer of $\#$, which we call the \emph{boundary} of the valid pattern. Each row of type $i$ will consist of only valid pattern of type $i$ with shared boundaries. Let $F$ be the set of forbidden $d+1$-dimensional words described below.  

\begin{description}
    \item[Building a grid:] For each $i$, forbid the words on hypercubes of length $2^{i}+1$ such that none of its rows is consistent with a valid pattern of type $i$.
  
    \item[Fixing the type:] 
    For each $i$, forbid the words that contain a valid pattern of type $i$ and something else inconsistent with a valid pattern of type $i$ in the same row.
    
    \item[Ensuring periodicity:] 
    For each $i$, forbid words with length-$2^{i-1}-1$-hypercubes of letters from $\Sigma$, surrounded by $\#$s, which are not periodic in one of the first $d$-dimensions with period length-$2^{i-1} \cdot |\Sigma|^{2^{d(i-1)}}$.
    
    \item[Recurrence of slices:] 
    For each $i$, forbid words that contain a valid pattern of type $i$ and anything other than the same valid pattern $2^{i}$ rows above it.
    \item[Subword-closed:] For each $i$, forbid words which contain a valid pattern of type $i$, which contains a subword of length $2^{i-1}-1$ not in a row of type $i-1$.
    
\end{description}

This c.e.\ set of forbidden words provides the scaffolding for the oracle shift.  Then $O_S = \pres{\Sigma, \#}{F \cup \cL^c(S)}^{d_S+1}$.

\subsubsection*{Proof of \cref{thm: second half of Ziegler}}

    Without loss of generality, suppose $\Sigma_S \cap \Sigma_T = \emptyset$.  Further, by considering $T^{(n)}$ or $S^{(n)}$ as appropriate (cf. \cref{def:new subshifts from old}), we may assume that $S$ and $T$ have the same dimension $d$.  Further, by \cref{char of computable subshifts}, we may assume $T \neq \emptyset$.

    Let $W_i$ witness that $\cL^c(S) \leq_e \cL^c(T)$ and $W_j$ witness that $\cL(S) \leq_e^1 \cL^c(T)$.  Then, for any $R$ with $T \sqsubseteq R$, we have $S \sqsubseteq R$ if and only if the following c.e.\ sets of implications hold:
    \[ I = \left\{ \bigwedge_{v \in D_n} v \in \cL^c(R) \limplies w \in \cL^c(R) \mid \langle w, n \rangle \in W_i \right \} \]
    \[ J = \left\{ \bigwedge_{v \in D_n} v \in \cL^c(R) \land D_m^1 \subseteq L(R) \limplies w \in \cL(R) \mid \langle w, n, m \rangle \in W_j  \right\}. \]

    We show that $I \cup J$ can be coded into a set of forbidden words in dimension $d + 1$.  The intuition is to modify the construction of the oracle shift described above by starting with a c.e.\ set of forbidden words that ensure the shift consists of interleaved $d$-dimensional rows from $O_{\pres{\Sigma_S}{}}$ and $O_{\pres{\Sigma_T}{}}$ (see \cref{fig:oracle}).  We can then add a further c.e.\ set of forbidden words that code the implications $I \cup J$  (see \cref{fig:axiom-tiles}). We assume that $O_{\pres{\Sigma_S}{}}$ and $O_{\pres{\Sigma_T}{}}$ have distinct safe symbols. That it, their alphabets are $\Sigma_S\cup\{\#\}$ and $\Sigma_T\cup\{\%\}$, respectively.  Further, recalling that our convention that the domains of words be rectangles, write $|w|$ for the longest side of $w$.

    Let $\cF$ be the set of forbidden words on alphabet $\Sigma_S\cup\Sigma_T\cup\{\#,\%\}$ as follows.
    
        \begin{description}
        \item[Alternation of $S$-rows and $T$-rows:] Forbid words that do not consist of a word in $\Sigma_{O_S}$ directly below (in the last dimension) a word in $\Sigma_{O_T}$. 
        \item[Synch the $S$-rows and the $T$-rows:] Forbid words to ensure that $S$-rows of type $i$ sit directly above $T$-rows of type $i$.
        \item[Oracle for $T$:] The forbidden words that ensure the $T$-rows are configurations of $O_T$ (see \cref{sec:oracle}).
        \item[Oracle for $S$:] Forbidden words to ensure that the $S$-rows are of the form of configurations of an oracle shift in the alphabet $\Sigma_{O_S}$ (see \cref{sec:oracle}).
        \item[Coding $I$:]\label{forb:I} For each $\langle w,u \rangle \in I$, forbid all configurations containing $w$ on an $S$-row for words of type $ \geq \log_2(|w|+1)$ but no element of $D_u$ in the corresponding places on the $T$-rows (see \cref{fig:axiom-tiles}).
        \item[Coding $J$:]\label{forb:J} Forbid all configurations that have no $w$ on an $S$-row in the appropriate type, no $v \in D_n$ in the corresponding places in the oracle rows and which do have $D_m^1$ in the corresponding place in the oracle rows  (see \cref{fig:axiom-tiles}).
    \end{description} 

    \begin{wrapfigure}{r}{0.45\textwidth}
        \includegraphics[width=.45\textwidth,alt={A window of the combined oracle shift.}]{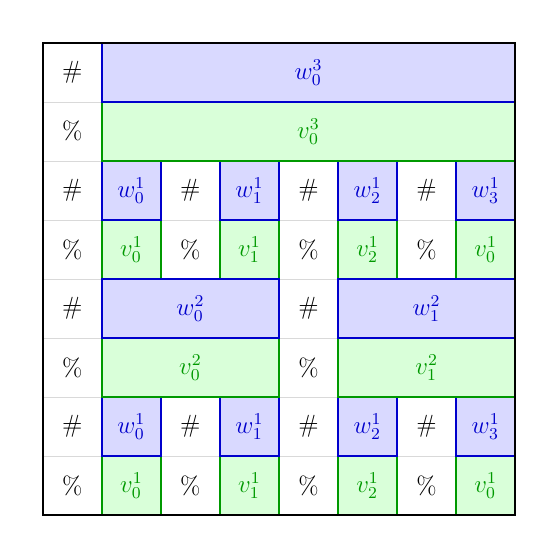}
        \caption{A window of the combined oracle shift.  The green words are from $\cL(T)$ while the blue are from $\cL(S)$.  The superscript on each word represents the \emph{type} of the row.  By subword closure, all the subwords of each word of length $2^i-1$ appear in the rows of type $i-1$.}
        \label{fig:oracle}
   \end{wrapfigure}

    Note that, while in the coding of $I$ and $J$ we are forbidding configurations (instead of words), thus can be accomplished by finite patterns because of the horizontal periodicity of $O_T$ and $O_S$, and the regularity of the rows of type $i$. In fact, for each implication $\langle w,u\rangle\in W_i$ or $ \langle w,n,m\rangle \in W_j$ it is more than enough to look at hypercubes of length $2^{l} \cdot |\Sigma|^{2^{d(l)}}$, where the longest word in the implication has length $l$. Moreover, observe that $\cF$ is c.e.

    Let $Q = \pres{\Sigma_S \cup \Sigma_T \cup \{\#,\%\}}{\cF}^{d+1}$.  Then $Q$ is effectively closed, so by Theorem 1 of \cite{durand_effective_2010} or Theorem 3 of \cite{aubrun_simulation_2013}, there is a SFT $R$ in $d+2$ dimensions and a factor map $f: \Sigma_R^{\Z^{d+2}} \to \Sigma^{\Z^{d+2}}_Q$ such that $Q = f(R) \restriction_{d+1}$.  Let $B$ be the finite set of forbidden words of $R$ and $\hat f$ a finite block code for $f$.  We claim that $S$ is finitely determined by $(\Sigma_R, d+2, \emptyset, B, \hat f)$ over $T$.
    
    Indeed, suppose that $P$ is a subshift in dimension $d+2$ and alphabet $\Sigma_R$ with $B \subseteq \cL^c(P)$.  Let $P_0 = f(P) \restriction_{d+1}$ and assume $T \sqsubseteq P_0$.  Then, $P_0$ has more forbidden words than $f(R)$.  In particular, it contains the forbidden words above, which force $P_0$ to be a subset of the oracle shift.  Let $w \in \cL(S)$.  Then there is an axiom $\langle w, n, m \rangle \in W_j$ such that $D_n \subseteq \cL^c(T)$ and $D^1_m \subseteq \cL(T)$.  Let $D_m^1 = \{v\}$.  Since $T \sqsubseteq f(P)$, take a configuration $x$ containing $v$, which must not contain any word in $D_n$.  Thus, by forbidden words coding $J$, $w$ must appear in $x$.

    \begin{figure}
        \centering
        \includegraphics[width=.49\textwidth,alt={One tile that exemplifies the coding of $I$}]{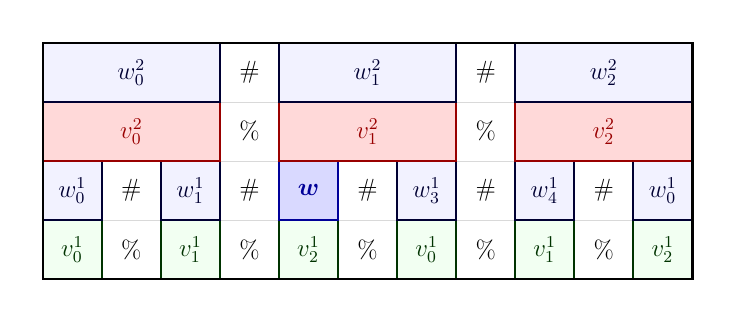}
        \includegraphics[width=.49\textwidth,alt={One tile that exemplifies the coding of $I$}]{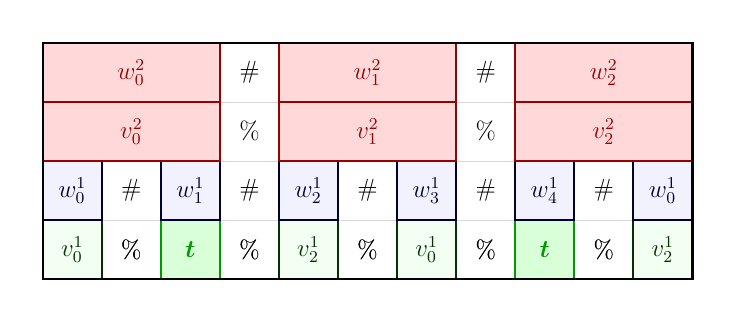}
        \caption{Example of tiles coding the axioms of $I$ (left) and $J$ (right). The $I$-tile represents and axiom $\langle w, u\rangle$ where $D_u$ only has words of type 2 and are not any $v^2_i$ (red row). The $J$-tile represents and axiom $\langle w, u, v\rangle$ where $w$ and all the words in $D_u$ have type 2 and do not appear in the corresponding row (red) and $D^1_v=\{t\}$ with $t$ of type 1. }
        \label{fig:axiom-tiles}
    \end{figure}
    
    For the other direction, let $w \in \cL^c(S)$.  This is witnessed by an enumeration axiom of the form $\langle w,u \rangle \in W_i$, where $D_u \subseteq \cL^c(T)$.  Then, \emph{on every configuration} of $P_0$, $D_u \subseteq \cL^c(x)$.  Since, since $P_0$ is a subset of the oracle shift, and forbids the words coding $I$, it must be that $w$ does not appear on any configuration, and hence $w \in \cL^c(P_0)$.

    This completes the proof of \cref{thm: second half of Ziegler}.

\begin{remark}
    It is worth noting that where Ziegler reduction features in this argument.  Indeed, if we required more than the single piece of negative information in $\leq_e^1$, then the proof would not go through.

    More precisely, it is important that the Ziegler axioms can be formulated as \emph{positive, quantifier-free Horn formulas}; i.e., formulas of the form $\bigwedge_{u \in D} L^c(u) \implies L^c(w)$.
    The key use of this fact is in the coding of $J$: each configuration in the oracle is set up to contain a subset of the forbidden words of $T$.  However, it is not necessarily the case that if $u$ and $v$ are both words in $\cL(T)$, that they must appear on the \emph{same} configuration.  For example, after passing through the factor map, it is possible that in $P_0$, every word containing both $u$ and $v$ is forbidden, even though $u,v \in \cL(P_0)$.  In this case, if we needed to see \emph{both} words in order to verify that $w \in \cL(S)$, it could be that $w \notin \cL(P_0)$. 
\end{remark}

\section{Existential closure}

Ziegler originally introduced $\leq^*$ in the context of \emph{existentially closed groups}.  
Say that a group $M$ is \emph{existentially closed} (ec) if for any quantifier-free $\varphi(\tuple{x}, \tuple{m})$, where $\tuple{m} \subseteq M$, if there is $N \geq M$ with $N \models \exists \tuple{x} \varphi$, then $M \models \exists \tuple{x} \varphi$.

Then Ziegler showed in \cite{ziegler_algebraisch_1980} that for finitely generated groups $G$ and $H$, $W(G) \leq^* W(H)$ iff $G$ embeds into every ec group that $H$ embeds into iff $G$ is $\exists_1$-isolated over $H$ (\cref{def:exists-isolated}; see also \cite[Chapter 4]{NakidCorderoComputability2026} for more detail on the background).  Thus there is an equivalence between (1) relative effective content, (2) relative embeddability in an algebraic structure, and (3) relative $\exists_1$-definability for groups.

Notably, Belegradek \cite{belegradek_higmans_1996} gave conditions for Ziegler's theorem to hold in a general algebraic context. In \cref{Zieglers thm full}, we translate the equivalence between (1) and (3) into the theory of subshifts.  In this section, we introduce a notion of existential closure for subshifts, namely being \emph{existentially closed for consistent systems}, in parallel to the definition for groups (\cref{def:existentially closed for consistent systems}).  

Subshifts which are existentially closed for consistent systems have some commonality with existentially closed groups (for example, both classes are amenable to model-theoretic forcing arguments), but the class nonetheless fails to satisfy an analogue of Ziegler's Theorem---indeed, there are no subshifts $S$ such that $S \sqsubsetsim M$ for every $M$ which is existentially closed for consistent systems.

\begin{definition}
    Let $U$ be a set of words.  We write $U(\tuple{a})$ to denote that the range of each word in $U$ is contained in $\tuple a$.  Further, we write $U(\tuple x, \tuple a)$ to highlight that $\tuple a$ comes from some given alphabet $\Sigma$ and $\tuple x$ comes from an alphabet disjoint from $\Sigma$, interpreted as ``variables''.

    Let $T$ be a subshift.  Then $\exists \tuple x\ (L(G(\tuple x, \tuple a)) \land L^c(B(\tuple x, \tuple a)))$ is \emph{consistent with $T$} if $\tuple a \subseteq \Sigma_T$ and there is $S \sqsupseteq T$ with $\tuple b \subseteq \Sigma_S$ such that $S \models L(G(\tuple b, \tuple a)) \land L^c(B(\tuple b, \tuple a))$.  In this case we say that $L(G(\tuple x, \tuple a)) \land L^c(B(\tuple x, \tuple a))$ is \emph{satisfied} in $S$.  

    If there is $T$ with alphabet $\Sigma$ which satisfies $L(G(\tuple x, \tuple a)) \land L^c(B(\tuple x, \tuple a))$, then say that $(G(\tuple x, \tuple a), B(\tuple x, \tuple a))$ is a consistent pair for $\Sigma$.
\end{definition}

\begin{remark}
    The alphabet $\Sigma$ is important in this last condition: $(\{ a \}, \{ aa \})$ is not a consistent pair for $\Sigma = \{a\}$ but is for $\Sigma = \{a,\#\}$, as witnessed by the set of strings in $\{a, \#\}^\Z$ with no repeated $a$s.
\end{remark}

The following proposition shows that it is straightforward to check if $L(G(\tuple x, \tuple a)) \land L^c(B(\tuple x, \tuple a))$ is consistent with $T$.

\begin{proposition}\label{sufficient condition for consistency}
    Let $T$ be a subshift.  Then $\exists \tuple x\ (L(G(\tuple x, \tuple a)) \land L^c(B(\tuple x, \tuple a)))$ is consistent with $T$ iff $G \cap \cL^c(T) = \emptyset$ and for every $w \in G$ and $v \in B$, $w$ does not contain $v$ as a subword.  Further, if $\exists \tuple x\ (L(G(\tuple x, \tuple a)) \land L^c(B(\tuple x, \tuple a)))$ is consistent with $T$, then it can be satisfied in a subshift in the same dimension as $T$.

    Moreover, $(G(\tuple x, \tuple a), B(\tuple x, \tuple a))$ is a consistent pair for $\Sigma$ with $|\Sigma \setminus \tuple a| > 1$ if and only if for every $w \in G$ and every $v \in B$, $w$ does not contain $v$ as a subword.
\end{proposition}

\begin{proof}
    Suppose that $G$ and $B$ satisfy the hypothesis of the proposition.
    Let $\Sigma = \Sigma_T \cup \{\#\} \cup \{ \tuple b\}$, where $|\tuple b| = |\tuple x|$.  Then consider the configuration $y$ consisting of, for each $w(\tuple x, \tuple a) \in G$, the word $w(\tuple b, \tuple a)$ surrounded by $\#$.  For example, in one dimension, if $G = \{ a_0x, a_1xa_0 \}$, then $y$ is the configuration
    \[ \cdots \#\#\#a_0b\#a_1ba_0\#\#\#\cdots \]
    Let $S$ be the least subshift in alphabet $\Sigma$ which is a superset of $T \cup \{y\}$.  Notice $S$ has dimension $d_T$, and it is evident from the construction that $T \sqsubseteq S$.

    The same argument works for the second part of the claim.
\end{proof}

\begin{definition}
\label{def:existentially closed for consistent systems}
	A subshift $M$ in alphabet $\Sigma$ is \emph{existentially closed for consistent systems (in dimension $d$)}, or \emph{ecfcs}, if for every pair of finite sets $U(\tuple{x}, \tuple{a})$ and $V(\tuple{x}, \tuple{a})$ of words in $(\Sigma \cup X)^{*d}$, if $\exists \tuple x\ (L(U(\tuple x, \tuple a)) \land L^c(V(\tuple x, \tuple a)))$ is consistent with $M$, then $L(U(\tuple x, \tuple a)) \land L^c(V(\tuple x, \tuple a))$ is satisfied in $M$.
\end{definition}

\begin{remark}
    We will show below that if $M$ is existentially closed for consistent systems, then $\Sigma_M$ must be infinite.  From this we lose the compactness of the space $\Sigma^{\Z^d}$, which means that many classical results, like the Curtis--Hedlund--Lyndon Theorem, are not available in this setting.
\end{remark}

We start by verifying that existentially closed subshifts for dimension $d$ exist.

\begin{theorem}\label{ec for consistent systems exist}
    For every $d$, there is a subshift that is ecfs in dimension $d$.
\end{theorem}

The proof is a simple forcing argument.  Because of its simplicity and flexibility, we expect that this construction may have other applications.  As examples, we present some properties of subshifts which are ecfcs and which are obtained by straightforward modifications of the following proof.

\begin{proof}
    Let $\Sigma=\{a_i\mid i<\omega\}$ be an infinite alphabet 
    and let $\{F_k\mid k<\omega\}$ be an effective enumeration of the finite subsets of $\left( \Sigma \cup X \right)^{*d}$, where $X$ is an infinite subset, disjoint from $\Sigma$, thought of as ``variable letters''.  
    
    We build by finite approximation sets $G$, $B$ of words in $\Sigma$ such that $G \subseteq \cL(\pres{\Sigma}{B})$.  $\pres{\Sigma}{B}$ will be the required subshift.

    \textit{Stage 0}: $G_0 = \varnothing$; $B_0 = \varnothing$.

    \textit{Stage s + 1}: Let $s = \langle j,k \rangle$ and consider the pair $(F_j, F_k)$.  If there is $T$ with $\Sigma_T \supseteq \Sigma$ that satisfies $L(G_s(\tuple a) \cup F_j(\tuple x, \tuple a)) \land L^c(B_s(\tuple a) \cup F_k(\tuple x, \tuple a))$, let $\tuple b \subseteq \Sigma \setminus \tuple a$ and set $G_{0,s+1} = G_s(\tuple a) \cup F_j(\tuple b, \tuple a)$ and $B_{s+1} = B_s(\tuple a) \cup F_k(\tuple b, \tuple a)$.  Then, let $G_{s + 1}$ be obtained from $G_{0,s + 1}$ by adding, for each $w \in G_{0,s + 1}$ a word $w^+$ in $\Sigma$ that contains $w$ as a subword such that the taxicab metric from any letter in $w$ to the edge of the word is $\geq s$ and such that $w^+$ does not contain as a subword any word from $B_{s+1}$.  We check below that, for every $s$, $G_{0,s+1} \subseteq \pres{\Sigma}{B_{s+1}}^d$, so, looking at a point of $\pres{\Sigma}{B_{s+1}}^d$, we may find words with arbitrarily large domains containing $w$ and not containing any element of $B_{s+1}$; hence we will always be able to find such $w^+$.

    This ends the construction.  Let $G = \bigcup_{s < \omega} G_s$ and $B = \bigcup_{s < \omega} B_s$, and let $M = \pres{\Sigma}{B}^d$.
    
    \smallskip

    We now verify that $G = \cL(M)$ and that $M$ is ecfcs.

    We start by showing that for each $s$, $G_s \subseteq \cL(\pres{\Sigma}{B_s}^d)$.  This is evidently true for $s = 0$.  Suppose that $G_s \subseteq \pres{\Sigma}{B_s}^d$.  If there is no $T$ with $\Sigma_T \supseteq \Sigma$ that satisfies $(L(G_s(\tuple a) \cup F_j(\tuple x, \tuple a) \land L^c(B_s(\tuple a) \cup F_k(\tuple x, \tuple a))$, then $B_{s+1} = B_s$ and so $G_s \subseteq \pres{\Sigma}{B_{s+1}}$ and we may pick $w^+$ so that $G_{s+1} \subseteq \cL(\pres{\Sigma}{B_{s+1}})$.  If there is such $T$, then by \cref{sufficient condition for consistency}, $(G_{0,s+1}, B_{s+1})$ is a consistent pair for $\Sigma$, so $G_{0,s+1} \subseteq \cL(\pres{\Sigma}{B_{s+1}})$.  Thus, we may ensure that $G_{s+1} \subseteq \cL(\pres{\Sigma}{B_{s+1}})$.

    Now we verify that $G = \cL(M)$.  Let $w \in G$.  Then $w \in G_s$ for some $s < \omega$.  Thus, by the second part of each stage, we obtain a sequence of words such that $w$ is a subword of $w_r$ is a subword of $w_t$ for every $s \leq r \leq t$.  Then, since each $w_t$ does not contain a subword from $B_t$, the limit word (obtained by fixing some element of $w$ at the origin) is in $\pres{\Sigma}{B}$.  On the other hand, let $w \in \cL(M)$ and $s = \langle j,k \rangle$ where $F_j = \{w\}$ and $F_k = \emptyset$.  Then, $M$ witnesses that $L(G_s \cup F_j) \land L^c(B_s)$ is consistent, so $w \in G_{s+1} \subseteq G$.

    Finally, we check that $M$ is ecfcs.  Let $U(\tuple{x}, \tuple{a})$ and $V(\tuple{x}, \tuple{a})$ be finite sets of words with $\tuple a \subseteq \Sigma$ and such that $\exists \tuple x\ (L(U) \land L^c(V))$ is consistent with $M$, as witnessed by $T \sqsupseteq M$.  Let $j,k$ be such that $F_j = U$ and $F_k = V$.  Then $T$ satisfies $(L(G_s(\tuple a) \cup F_j(\tuple x, \tuple a) \land L^c(B_s(\tuple a) \cup F_k(\tuple x, \tuple a))$, so $G_{s+1} \supseteq G_s(\tuple a) \cup F_j(\tuple b, \tuple a)$ and $B_{s+1} = B_s(\tuple a) \cup F_k(\tuple b, \tuple a)$.  Since $G_{s+1} \subseteq G \subseteq \cL(M)$, we have that $M$ satisfies $L(U) \land L^c(V)$, as required.
\qed\end{proof}

The next proposition states that any subshift is $\sqsubseteq$-embedded in one that is existentially closed for consistent systems.

\begin{proposition}
    For any subshift $S$, there is a subshift $M \sqsupseteq S$ with $M$ ecfcs.
\end{proposition}

\begin{proofsketch}
    Run the argument as in the proof of \cref{ec for consistent systems exist}, but start with $G_0 = L(S)$ and $B_0 = L^c(S)$.
\qed\end{proofsketch}

\begin{proposition}
    Let $M$ be existentially closed for consistent systems in dimension $d$.  Then $\Sigma_M$ is infinite.
\end{proposition}

\begin{proof}
    Suppose that $\Sigma_M = \{a_0, \ldots, a_{n-1}\}$ and, without loss of generality, that $a_i \in \cL(M)$ for all $i < n$.  Then, by \cref{sufficient condition for consistency}, $\exists x\ (L(a_0x) \land L^c( xa_0, \ldots, xa_{n-1}))$ is consistent with $M$, and hence satisfied in $M$, say by $b$.

    However, $b$ cannot be $a_i$ for any $i$: for every $i$, because $a_i \in \cL(M)$, $a_i$ satisfies one of $\{ xa_0, \ldots, xa_{n-1}\}$.
\qed\end{proof}

A similar argument implies that existentially closed shifts for consistent systems are transitive.  

\begin{proposition}
    Let $M$ be existentially closed for consistent systems in dimension $d$.  Then for any $v,w \in L(M)$, there is $x \in M$ with $v,w \in L(x)$.
\end{proposition}

We note that the following argument fundamentally relies on the convention that words are connected.

\begin{proof}
    (We work in dimension 1 for simplicity, but the argument evidently generalises.)  Let $v,w \in \cL(M)$.   Then, by \cref{sufficient condition for consistency}, $\exists x\ L(wxv)$ is consistent with $M$ and hence satisfied in $M$.
\qed\end{proof}

However, this notion of existential closure does not satisfy an analogue of Ziegler's Theorem.  We use $\sqsubsetsim$ rather than $\sqsubseteq$ in this theorem, otherwise the statement is trivially true, by choosing $M$ and $N$ ecfcs's  but with $\Sigma_M \cap \Sigma_N = \emptyset$.

\begin{theorem}\label{no universal substructure for ecfcs}
    There is no subshift $S \neq \emptyset$ such that for every subshift $M$ that is existentially closed for consistent systems in dimension $d$, $S \sqsubsetsim M$.
\end{theorem}

\begin{proof}
We implicitly fix throughout a dimension $d$.  We start by showing that any $S$ which satisfies the condition of the theorem is ``strongly finitely determined''.  

Suppose that $S \sqsubsetsim M$ for every ecfcs $M$ and let $\Sigma_S = \{a_0, \ldots, a_{n-1}\}$ be an ordering of $\Sigma_S$. Let $\Sigma=\{c_i\mid i<\omega\}$ be an infinite alphabet and let $\{F_k\mid k<\omega\}$ again be an effective enumeration of the finite subsets of $\left( \Sigma \cup X \right)^{*d}$, where $X$ is an infinite subset, disjoint from $\Sigma$.   If $|\tuple c| = |\Sigma_S|$, we write $S_{\tuple c}$ for the subshift on alphabet $\tuple c$ obtained by replacing $a_i$ with $c_i$.
\begin{claim}
    There is an alphabet $\Delta$, a pair $(G,B)$ of finite sets of words in $\Delta$, and some $\tuple c \subseteq \Sigma$ such that for every $\Delta' \supseteq \Delta$, and every $T$ with alphabet $\Delta'$ if $T \models L(G) \land L^c(B)$, then $S_{\tuple c} \sqsubseteq T$.
\end{claim}

\begin{claimproof}
    Suppose $S$ does not satisfy the conclusion of the claim.  Then we will build $M$ by modifying the proof of \cref{ec for consistent systems exist} with $S \not\sqsubsetsim M$.  Let $(\tuple{c}_e)_{e < \omega}$ be an enumeration of the tuples of length $|\Sigma_S|$ from $\Sigma$.
    
    At even stages $s$, let $(G_s,B_s)$ be the stage $s$-approximation to $M$.  As in the previous proof, add the $\frac s2$th pair $(F_j, F_k)$ if it is consistent to do so; and do nothing otherwise.  Further add, for every $w \in G$, a word $w^+$ as before.  This ensures that $M$ is existentially closed for consistent systems.

    At odd stages $s$, we ensure that the $(s-1)/2$th tuple $\tuple c \subseteq \Sigma$ does not give rise to a full restriction of $M$ equal to $S_{\tuple c_{(s-1)/2}}$.  Let $(G_s,B_s)$ be the approximation to $M$ at stage $s$ and let $\Sigma_s$ be the symbols used so far.  By assumption, there is some $\Delta \supseteq \Sigma_s$ and some $T$ with alphabet $\Delta$ such that $T \models L(G_s) \land L^c(B_s)$ but $S_{\tuple c_{(s-1)/2}} \not\sqsubseteq T$.  Thus, we may take a word $w \in \Delta^{*d}$ such that $(G_s \cup \{w\}, B_s)$ is a consistent pair for $\Delta$ and $w \in \cL^c(S_{\tuple{c}_{(s-1)/2}})$ or $(G_s, B_s \cup \{w\})$ is a consistent pair for $\Delta$ but $w \in \cL(S_{\tuple c_{(s-1)/2}})$.  This ensures that $S \not\sqsubsetsim M$ via $\tuple c_{(s-1)/2}$.
    \renewcommand{\squareforqed}{$\dashv$}
    \qed
\end{claimproof}

We finish the proof by showing that no $S \neq \emptyset$ satisfying the claim exists.  Indeed, let $S$ be a subshift in alphabet $\Sigma_S$ and let $(G,B)$ be a consistent pair for some alphabet $\Lambda$.  

Fix $\Delta \supseteq \Lambda \cup \{\#\}$.  For each of the finitely many $\tuple{c} \subseteq \Lambda$ with $|\tuple{c}| = |\Sigma_S|$, choose $w_{\tuple c} \in \cL(S_{\tuple c})$ such that, for every $v \in G$, $w_{\tuple c}$ is not a subword of $v$.  Here we use that $S \neq \emptyset$, so $\cL(S_{\tuple c})$ is infinite.  Let 
\[B^+ = B \cup \{w_{\tuple{c}} \mid \tuple c \subseteq \Lambda\; \land \;|\tuple{c}| = |\Sigma_S| \}.\]
Then by \cref{sufficient condition for consistency}, $(G, B^+)$ is a consistent pair for $\Delta$.  Further, the construction of $B^+$ evidently guarantees that for every $T$ with $T \models L(G) \land L^c(B^+)$ and every $\tuple c \subseteq \Lambda$, $S_{\tuple c} \not\sqsubseteq T$.
\qed\end{proof}

\begin{remark}
One important concept in group theory that has no natural analogue in subshift theory is that of a finitely generated substructure.  It would be natural, for example, to look at the alphabet slices $M \restriction_\Delta$ for $\Delta \subseteq_{fin} \Sigma_M$.  However, in general, ecfcs subshifts may not have any nontrivial alphabet slices:
    
    In fact, a construction like the one in \cref{ec for consistent systems exist} \emph{cannot} guarantee that there are nontrivial slices.  For instance, one could, at odd stages $s = 2t+1$, add a word $v^+$ to $B$ for each $v \in \cL(\pres{\tuple c}{B_{s+1}})$ such that $v^+$ contains $v$ as a subword and $B^+ = B \cup \{v^+ \mid v \in \cL(\pres{\tuple c}{B_{s+1}})\}$ forms a consistent pair with $G_s$, this guarantees that $M \restriction_{\tuple c} = \emptyset$.

\smallskip

    This underscores a key tension in subshifts between any notions of existential closure and finitary definability.  On the one hand, the notions of definability, as exemplified in the previous proof (via \cref{sufficient condition for consistency}), are very sensitive to the ambient alphabet.  However, forcing constructions tend to produce subshifts with various mixing properties; in particular, with very little ``local structure''.
\end{remark}

As a final application, we observe that the technique of the proof of \cref{no universal substructure for ecfcs} together with the construction of Hodges in Theorem 4.1.5 in \cite{hodges_building_1985} allows one to show the following.
\begin{proposition}
    There are $2^{\aleph_0}$ ecfcs subshifts $(M_\alpha)_{\alpha < 2^\omega}$ on a countably infinite alphabet $\Sigma$ such that if $\alpha \neq \beta$, then $M_\alpha \not\sqsubsetsim M_\beta$.
\end{proposition}

Notice that this is optimal---any subshift in $\Sigma$ is determined by a subset of $\Sigma^{*d}$, of which there are continuum many.

\begin{proofsketch}
    The idea of this argument is to run the by-now standard construction of existentially closed subshifts for consistent systems, but, at each stage $s$, ensure that there at least $2^s$ different subshifts $(M_\sigma)_{\sigma \in 2^s}$.  This will be done by running the construction on a binary tree. One further ensures that stage $s$ that for every pair of tuples $\tuple c$ and $\tuple d$ from $\Sigma$ and every pair of approximations $M_\sigma$ and $M_\tau$, there is a word $w_{\tuple c, \tuple d}(\tuple x)$ such that $w_{\tuple c, \tuple d}(\tuple c) \in \cL^c(M_\sigma)$ and $w_{\tuple c, \tuple d}(\tuple d) \in \cL(M_\tau)$, or vice versa.  This ensures that, for any $\alpha \succeq \sigma$ and $\beta \succeq \tau$, we have $M_\alpha \not\sqsubsetsim M_\beta$ via a map extending $\tuple c \mapsto \tuple d$.
\qed\end{proofsketch}

\begin{credits}
\subsubsection{\ackname} The first author was supported by the Secretaría de Ciencia, Humanidades, Tecnología e Innovación de México through scholarship number 836694.  Both authors would like to thank Mariya Soskova for many helpful conversations, and the anonymous referees for their detailed feedback.  The authors used AI tools to aid in the construction of their figures.

\subsubsection{\discintname}
The authors have no competing interests to declare that are
relevant to the content of this article.
\end{credits}

\bibliographystyle{plain}
\bibliography{references}

@inproceedings{jeandel_characterization_2019,
    address = {Dagstuhl, Germany},
    series = {{LIPIcs}},
    title = {A {Characterization} of {Subshifts} with {Computable} {Language}},
    volume = {126},
    language = {en},
    booktitle = {36th {International} {Symposium} on {Theoretical} {Aspects} of {Computer} {Science}},
    publisher = {Schloss Dagstuhl},
    author = {Jeandel, Emmanuel and Vanier, Pascal},
    year = {2019},
    pages = {40:1--40:16},
}

@article{hochman_dynamics_2009,
    title = {On the dynamics and recursive properties of multidimensional symbolic systems},
    volume = {176},
    language = {en},
    number = {1},
    journal = {Inventiones mathematicae},
    author = {Hochman, Michael},
    month = apr,
    year = {2009},
    pages = {131--167},
}

@article{aubrun_simulation_2013,
    title = {Simulation of {Effective} {Subshifts} by {Two}-dimensional {Subshifts} of {Finite} {Type}},
    volume = {126},
    language = {en},
    number = {1},
    journal = {Acta Applicandae Mathematicae},
    author = {Aubrun, Nathalie and Sablik, Mathieu},
    month = aug,
    year = {2013},
    pages = {35--63},
}

@incollection{ziegler_algebraisch_1980,
    series = {Studies in {Logic} and the {Foundations} of {Mathematics}},
    title = {Algebraisch {Abgeschlossene} {Gruppen}},
    volume = {95},
    booktitle = {Word {Problems} {II}},
    publisher = {Elsevier},
    author = {Ziegler, Martin},
    year = {1980}
}

@inproceedings{jeandel_enumeration_2017,
    title = {Enumeration reducibility in closure spaces with applications to logic and algebra},
    booktitle = {{LICS}},
    author = {Jeandel, Emmanuel},
    year = {2017}
}

@incollection{durand_effective_2010,
    address = {Berlin, Heidelberg},
    title = {Effective {Closed} {Subshifts} in {1D} {Can} {Be} {Implemented} in {2D}},
    language = {en},
    booktitle = {Fields of {Logic} and {Computation}: {Essays} {Dedicated} to {Yuri} {Gurevich} on the {Occasion} of {His} 70th {Birthday}},
    publisher = {Springer},
    author = {Durand, Bruno and Romashchenko, Andrei and Shen, Alexander},
    year = {2010}
}

@book{hodges_building_1985,
  title={Building Models by Games},
  author={Hodges, W.},
  series={London Mathematical Society Student Texts},
  year={1985},
  publisher={Cambridge University Press}
}

@article{belegradek_higmans_1996,
    title = {Higman's {Embedding} {Theorem} in a {General} {Setting} and {Its} {Application} to {Existentially} {Closed} {Algebras}},
    volume = {37},
    number = {4},
    journal = {Notre Dame Journal of Formal Logic},
    publisher = {Duke University Press},
    author = {Belegradek, Oleg V.},
    year = {1996}
}

@article{rips_characterization_82,
    author = {Rips, E.},
    title = {Another Characterization of Finitely Generated Groups with a Solvable Word Problem},
    journal = {Bulletin of the London Mathematical Society},
    volume = {14},
    number = {1},
    pages = {43-44},
    year = {1982}
}

@phthesis{NakidCorderoComputability2026,
  title = {The Computability of Asymmetric Information},
  author = {Nakid Cordero, Antonio},
  year = {2026},
  school = {University of Wisconsin--Madison}
}

\end{document}